\documentclass[12pt]{amsart}
\usepackage{amsthm}

\usepackage[latin1]{inputenc}
\usepackage{subfigure}
\usepackage{color}
\usepackage{amsmath}
\usepackage{amsthm}
\usepackage{amstext}
\usepackage{amssymb}
\usepackage{amsfonts}
\usepackage{graphicx}
\usepackage{young}
\usepackage{multicol}
\usepackage{mathrsfs}
\usepackage[all]{xy}
\usepackage{tikz}
\usepackage{mathtools}
\usepackage{tabularx}
\usepackage{array}
\usepackage{commath}
\usepackage{ytableau}

\theoremstyle{plain}
\theoremstyle{definition}
\newtheorem{theorem}{Theorem}[section]

\newtheorem{remark}[theorem]{Remark}
\newtheorem{lemma}[theorem]{Lemma}

\newtheorem{example}[theorem]{Example}
\newtheorem{proposition}[theorem]{Proposition}
\newtheorem{corollary}[theorem]{Corollary}

\DeclareMathAlphabet{\mathpzc}{OT1}{pzc}{m}{it}

\usepackage{amssymb,amsmath,tabularx,graphicx}
\usepackage{tikz}
\usetikzlibrary[calc,intersections,through,backgrounds,arrows,decorations.pathmorphing]

\begin{document}

\title{Greene--Kleitman invariants for Sulzgruber insertion}

\date{}
\author{Alexander Garver}
\address{Laboratoire de Combinatoire et d'Informatique Math\'ematique, Universit\'e du Qu\'ebec \`a Montr\'eal}
\email{alexander.garver@lacim.ca}

\author{Rebecca Patrias}
\address{Laboratoire de Combinatoire et d'Informatique Math\'ematique,
Universit\'e du Qu\'ebec \`a Montr\'eal}
\email{patriasr@lacim.ca}

\begin{abstract}
R. Sulzgruber's rim hook insertion and the Hillman--Grassl correspondence are two distinct bijections between the reverse plane partitions of a fixed partition shape and multisets of rim-hooks of the same partition shape. It is known that Hillman--Grassl may be equivalently defined using the Robinson--Schensted--Knuth correspondence, and we show the analogous result for Sulzgruber's insertion. We refer to our description of Sulzgruber's insertion as diagonal RSK. As a consequence of this equivalence, we show that Sulzgruber's map from multisets of rim hooks to reverse plane partitions can be expressed in terms of Greene--Kleitman invariants.
\end{abstract}
\maketitle
\tableofcontents

\section{Introduction}

Reverse plane partitions are prominent combinatorial objects with connections to areas like symmetric functions and representation theory (see for example \cite{stanley1971theory}). Their generating function was discovered by Stanley and is as follows, where $h(u)$ denotes the hook length of the cell $u$ in partition $\lambda$ and $|\pi|$ denotes the sum of the entries in reverse plane partition $\pi$.
\begin{theorem}\cite{stanley1971theory}\label{thm:genfunction}
The generating function for reverse plane partitions of shape $\lambda$ with respect to the sum of its entries is 
\[\sum_{\pi}q^{|\pi|}=\prod_{u\in\lambda}\frac{1}{1-q^{h(u)}}.\]
\end{theorem}

The first bijective proof of this generating function was found by Hillman and Grassl in \cite{hillman1976reverse}. The authors give a bijection between nonnegative integer arrays of shape $\lambda$---representing multisets of rim hooks of $\lambda$---and reverse plane partitions of $\lambda$, which is now known as the \textit{Hillman--Grassl correspondence}. This correspondence has since been well studied, for example by Gansner in \cite{gansner1981hillman} and by Morales, Pak, and Panova in \cite{morales2015hook}. In particular, Gansner relates Hillman--Grassl to the Robinson--Schensted--Knuth (RSK) correspondence and shows that the correspondence has the same Greene--Kleitman invariants as RSK. 

Recently, Sulzgruber defined a new bijection between multisets of rim hooks of $\lambda$ and reverse plane partitions of $\lambda$ that takes the form of a rim hook insertion algorithm, which we call \textit{Sulzgruber insertion} \cite{sulzgruber2016building}. It is easy to check that Sulzgruber insertion is distinct from Hillman--Grassl, and so his correspondence gives an alternative bijective proof of Theorem~\ref{thm:genfunction}.

In this paper, we explicitly phrase Sulzgruber insertion in terms of RSK insertion with the goal of proving that Sulzgruber insertion can be expressed in terms of Greene-Kleitman invariants. Each diagonal of the reverse plane partition associated to a multiset of rim hooks encodes the Greene--Kleitman partition associated to a certain poset on a subset of the rim hooks--namely, the rim hooks whose support intersects the given diagonal. This idea is relevant to future work of the authors with Hugh Thomas that relates reverse plane partitions to the theory of quiver representations \cite{garver2017+}. %In particular, Sulzgruber insertion for rectangular shapes can be interpreted as a special case of a more general family of bijections for quivers of cominuscule type, and the idea of diagonal RSK is a key tool in this construction. 

The paper proceeds as follows. In Section~\ref{sec:prelim}, we set up notation and review the necessary combinatorial background related to tableaux, the RSK correspondence, and Knuth equivalence. Section~\ref{sec:Sulz} defines Sulzgruber insertion as well as its inverse. We also review the Hillman--Grassl correspondence. In Section~\ref{sec:diagonalRSK}, we define a rim hook insertion algorithm we call \textit{diagonal RSK}, and we show that diagonal RSK is equivalent to Sulzgruber insertion. Section~\ref{sec:HillmanGrasslRSK} reminds the reader how to phrase Hillman--Grassl in terms of RSK in the same spirit as our diagonal RSK and states the corresponding Greene--Kleitman invariants. We conclude by showing that Sulzgruber insertion has a Greene--Kleitman invariant in Section~\ref{sec:GK}. 
 
\section{Preliminaries}\label{sec:prelim}
\subsection{Reverse plane partitions and rim hooks}
A \textit{partition} is a finite, weakly decreasing sequence of positive integers $\lambda=(\lambda_1\geq\cdots\geq\lambda_k)$. Let $|\lambda|=\lambda_1+\cdots+\lambda_k$. To each partition, we associate a \textit{Young diagram}: a left-justified array of boxes where row sizes weakly decrease from top to bottom. Throughout this paper, we refer to both the partition and its Young diagram as $\lambda$. We count the rows of a Young diagram from top to bottom and the columns from left to right, and we refer to a box in a Young diagram by its row and column index. We let $\lambda'=(\lambda_1',\ldots,\lambda_k')$ denote the \textit{conjugate} or the \textit{transpose} of partition $\lambda$---the partition obtained by swapping the roles of the rows and columns in $\lambda$. Thus $\lambda_i'$ is the size of the $i$th column of $\lambda$.

We impose a partial ordering on the boxes of a Young diagram by defining $(i,j) \preceq (i^\prime,j^\prime)$ if $i \le i^\prime$ and $j \le j^\prime$; in other words, $(i,j) \preceq (i^\prime,j^\prime)$ if $(i,j)$ is weakly northwest of $(i^\prime,j^\prime)$. We say that a box $(i,j) \in \lambda$ is in \textit{diagonal} $d$ of $\lambda$ if $d = j - i$. With respect to the partial order $\prec$, each diagonal $d$ of $\lambda$ has a unique maximal element $\textbf{r}_d = (i,j)$, the southeasternmost entry in the diagonal. We refer to $\textbf{r}_d$ as a \textit{border box} of $\lambda$.

A \textit{rim hook} of $\lambda$ is a connected strip of border boxes in $\lambda$. To each box $(i,j)$ of $\lambda$, we associate the unique rim hook of $\lambda$ that has its southwesternmost box in column $j$ and northeasternmost box in row $i$. We denote the rim hook associated to box $(i,j)$ of fixed shape $\lambda$ by $h_{(i,j)}$.

A \textit{reverse plane partition} of shape $\lambda$ is an order-preserving map $\lambda\to\{0,1,2,\ldots\}$, i.e., a filling of the boxes of a Young diagram with nonnegative integers such that entries weakly increase across rows and down columns. An $\mathbb{N}$-\textit{tableau} is a filling of shape $\lambda$ with nonnegative integers. A \textit{semistandard Young tableau} of shape $\lambda$ is a filling of the boxes of $\lambda$ with positive integers such that entries weakly increase across rows and strictly increase down columns. A \textit{standard Young tableau} is a semistandard Young tableau where the entries $1,2,\ldots,|\lambda|$ each appear exactly once. See Figure~\ref{fig:rimhookrppexample}.

\begin{figure}[h!]
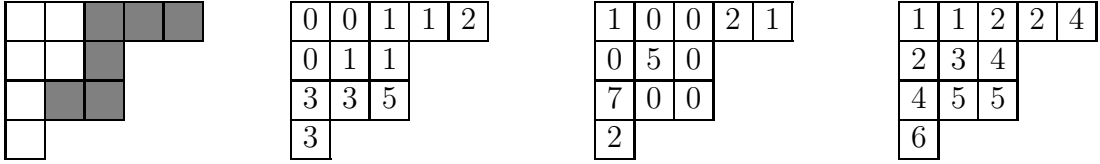

\ytableausetup{boxsize=.2in}
\begin{center}
\begin{ytableau}
$ $ & & *(gray) & *(gray) & *(gray) \\
$ $ & & *(gray) \\
$ $ & *(gray) & *(gray)\\
$ $
\end{ytableau}\hspace{.4in}
\begin{ytableau}
0 & 0 & 1 & 1 & 2 \\
0 & 1 & 1 \\
3 & 3 & 5\\
3
\end{ytableau}\hspace{.5in}
\begin{ytableau}
1 & 0 & 0 & 2 & 1\\
0 & 5 & 0 \\
7 & 0 & 0\\
2
\end{ytableau}\hspace{.5in}
\begin{ytableau}
1 & 1 & 2 & 2 & 4\\
2 & 3 & 4 \\
4 & 5 & 5\\
6
\end{ytableau}
\end{center}
\caption{From left to right, we have the rim hook $h_{(1,2)}$ inside shape $\lambda=(5,3,3,1)$, a reverse plane partition of shape $\lambda$, an $\mathbb{N}$-tableau of shape $\lambda$, and a semistandard Young tableau of shape $\lambda$.}\label{fig:rimhookrppexample}
\end{figure}

\subsection{RSK insertion}
There is a well-known bijective correspondence between words in the alphabet of positive integers and pairs consisting of a  
semistandard Young tableau and a standard Young tableau of the same shape called the Robinson--Schensted--Knuth (RSK) correspondence. 
We briefly review the correspondence here and refer the reader to \cite{stanley1999enumerative} for more information.

Given a word $w = w_1w_2 \cdots w_r$ in the positive integers, the RSK correspondence maps $w$ to a pair of tableaux via a row insertion algorithm consisting of inserting a positive 
integer into a tableau. The algorithm for inserting positive integer $k$ into a row of a semistandard tableau is as follows. If $k$ is 
greater than or equal to all entries in the row, add a box labeled $k$ to the end of the row. Otherwise, find the first $y$ in the row 
with $y>k$. Replace $y$ with $k$ in this box, and proceed to insert $y$ into the next row. To insert $k$ into semistandard tableau $P$, 
we start by inserting $k$ into the first row of $P$. To create the \textit{insertion tableau} of a word $w=w_1w_2\cdots w_r$, we 
first insert $w_1$ into the empty tableau, insert $w_2$ into the result of the previous insertion, insert $w_3$ into the result of the 
previous insertion, and so on until we have inserted all letters of $w$. We denote the resulting insertion tableau by $P(w)$. The insertion 
tableau will always be a semistandard tableau.

To obtain a standard Young tableau from $w$, we define the \textit{recording tableau}, $Q(w)$, of $w$ to be the tableau with the same shape as $P(w)$ and label $s$ in the box added to $P(w_1\cdots w_{s-1})$ during the insertion of $w_s$. For example, $w=14252$ has insertion tableau and recording tableau

\begin{center}
$P(w)=$
 \begin{ytableau}
  1 & 2 & 2 \\
  4 & 5
 \end{ytableau}\hspace{1in}
 $Q(w)=$
 \begin{ytableau}
  1 & 2 & 4 \\
  3 & 5
 \end{ytableau}\ .
\end{center}

Given a pair $(P,Q)$ of a semistandard and standard Young tableau of the same shape, one can easily reverse the procedure to recover $w$ by using $Q$ to determine which entry of $P$ to reverse insert at each step.

We next discuss a few important properties of RSK insertion that we will use later. The first is known as the Greene--Kleitman invariant.

\begin{theorem}[\cite{greene1976structure}]\label{thm:GreeneKleitman} Let $\delta_k(w)$ denote the largest cardinality of a disjoint union of $k$ nondecreasing  subwords of the word $w$ with $\delta_0(w)=0$, and let $\gamma_k(w)$ denote the largest cardinality of a disjoint union of $k$ strictly decreasing subwords of $w$ with $\gamma_0(w)=0$. Then $P(w)$ has shape $(\lambda_1,\ldots,\lambda_k)$ with 
\begin{eqnarray*}
\lambda_i&=&\delta_i(w)-\delta_{i-1}(w)\text{   and }\\
\lambda'_i&=&\gamma_i(w)-\gamma_{i-1}(w)
\end{eqnarray*} for all $i=1,\ldots,k$.
\end{theorem}

We will also need the following easy lemma, which follows from the fact that larger numbers never bump smaller numbers in RSK insertion. Let $w|_{[1,k]}$ denote the restriction of the word $w$ to the alphabet $[1,k]$ and $P_{[1,k]}$ denote the restriction of semistandard Young tableau $P$ to boxes containing $[1,k]$.

\begin{lemma}\label{lem:restriction}
We have that $P(w|_{[1,k]})=P(w)_{[1,k]}.$
\end{lemma}

Next, we define the \textit{bumping path} of RSK insertion of a positive integer into a semistandard Young tableau $T$ to be the set of cells of $T$ that are added or whose entries change during the insertion process. We order the list from top to bottom row. For example, the bumping path for the insertion of 1 into $P(14252)$ shown above is $(1,2)$, $(2,1)$, $(3,1)$. Let $(T\leftarrow i)$ denote the result of inserting $i$ into semistandard Young tableau $T$. The following lemma can be found, for example, in \cite{stanley1999enumerative}.

\begin{lemma}\label{lem:bumpsmoveleft}
\begin{enumerate}
\item[$(a)$] When we insert positive integer $k$ into semistandard Young tableau $T$, the bumping path moves weakly left.
\item[$(b)$] Let $T$ be a semistandard Young tableau and let $i\leq j$. Then the bumping path of $i$ in $(T\leftarrow i)$ lies strictly to the left of the bumping path of $j$ in $((T\leftarrow i) \leftarrow j)$.
\end{enumerate}

\end{lemma}

\subsection{Knuth equivalence}
We recall the Knuth equivalence relations on words in the alphabet of positive integers. The \textit{elementary Knuth relations} are given below, where $x,y,z$ are positive integers:
\begin{eqnarray*}
 yzx \equiv yxz &\text{ whenever }& x < y \leq z \\
 xzy \equiv zxy &\text{ whenever }& x \leq y < z.
\end{eqnarray*}

We then say that two words are \textit{Knuth equivalent} if one can be obtained from the other by a finite sequence of the elementary Knuth relations. For example, $221343\equiv 241233$ because of the sequence below, where the underlined triplets show to which three letters the equivalence was applied:
\[\underline{221}343 \equiv 212\underline{343} \equiv 21\underline{243}3 \equiv \underline{214}233 \equiv 241233.\]
One of the main motivations for Knuth equivalence is the connection to RSK insertion.

\begin{theorem}[\cite{knuth1970permutations}] Let $u$ and $v$ be two words. Then $P(u)=P(v)$ if and only if $u\equiv v$.
\end{theorem}

We define the \textit{reading word} of a tableau $T$, $r(T)$, to be the word obtained by reading the entries of $T$ row by row from left to right starting with the bottom row. For example, $r(T)=645523411224$ for the rightmost tableau $T$ in Figure~\ref{fig:rimhookrppexample}. The following fact is well known, and the interested reader may find a proof in \cite{stanley1999enumerative}.

\begin{theorem}\label{thm:insertionword}
For any word $w$, $r(P(w))\equiv w$.
\end{theorem}

We will use Theorem~\ref{thm:insertionword} in Section~\ref{sec:diagonalRSK} in the following way. Suppose we have a word $w$ on the set of positive integers $a_1<\ldots <a_n$. In addition, suppose that $w|_{[a_k,a_n]}$ is weakly increasing for some $k\in[1,n]$. It follows from the definition of the Knuth relations that $u|_{[a_k,a_n]}$ will be weakly increasing for any $u\equiv w$. In particular, $r(P(w))|_{[a_k,a_n]}$ will be weakly increasing. This implies that the boxes of $P(w)$ with entries in $[a_k,a_n]$ form a horizontal strip (i.e., no two of the boxes lie in the same column). We state this as a lemma.

\begin{lemma}\label{lem:horizontal strip}
Suppose $w$ is a word on positive integers $a_1<\cdots <a_k<\cdots< a_n$ and $w|_{[a_k,a_n]}$ is weakly increasing. Then the boxes of $P(w)$ containing $a_k,\ldots,a_n$ form a horizontal strip.
\end{lemma}

\subsection{The Hillman--Grassl correspondence}
We next review the Hillman--Grassl correspondence. We will start with a reverse plane partition $\pi$ of shape $\lambda$ and obtain a multiset of rim hooks of $\lambda$.

Let $\pi_0=\pi$. Scan the columns of $\pi_0$ from bottom to top and left to right to find the cell with the first nonzero entry, and call this cell $v_0$. We inductively form a lattice path $L_0$ in $\pi$ starting at $v_0$ and taking only north and east steps as follows. Suppose we have constructed the path up to some box $v\in\lambda$. 
\begin{itemize}
\item If there is a box directly above $v$ and its entry is equal to the entry in $v$, the lattice path travels to the box directly above $v$.
\item If there is not a box directly above $v$ or its entry is strictly less than the entry in $v$ and there is a box directly to the right of $v$, the path travels to the box directly right of $v$.
\item If there is not a box directly above $v$ or its entry is strictly less than the entry in $v$ and there is not a box directly to the right of $v$, the path ends at $v$.
\end{itemize}

Next, associate a rim hook $h_0$ of $\lambda$ to $L_0$ by taking the rim hook with southwesternmost and northeasternmost cells the same as the starting and ending cells of $L_0$. Define a new reverse plane partition $\pi_1$ to be the result of subtracting 1 from each entry of $\pi_0$ that intersects $L_0$.

Continue this process until $\pi_k$ is the reverse plane partition filled with all zeros. The multiset of rim hooks of $\lambda$ that corresponds to $\pi$ under Hillman--Grassl is then $\mathcal{M}=\{h_0,\ldots,h_{k-1}\}$.

\begin{example} Starting with reverse plane partition $\pi=\pi_0$, we end with multiset $\mathcal{M}$.
\[
\pi_0=
\begin{ytableau}
0 & 0 & 1 \\
1 & 2 & 2 \\
\mathbf{3} & \mathbf{4}
\end{ytableau}\hspace{.1in}
\pi_1=
\begin{ytableau}
0 & 0 & 1 \\
1 & 2 & 2 \\
\mathbf{2} & \mathbf{3}
\end{ytableau}\hspace{.1in}
\pi_2=
\begin{ytableau}
0 & 0 & 1 \\
\mathbf{1} & \mathbf{2} & \mathbf{2} \\
\mathbf{1} & 2
\end{ytableau}\hspace{.1in}
\pi_3=
\begin{ytableau}
0 & 0 & 1 \\
0 & 1 & 1 \\
0 & \mathbf{2}
\end{ytableau}\hspace{.1in}
\pi_4=
\begin{ytableau}
0 & 0 & \mathbf{1} \\
0 & \mathbf{1} & \mathbf{1} \\
0 & \mathbf{1}
\end{ytableau}\hspace{.1in}
\pi_5=
\begin{ytableau}
0 & 0 & 0 \\
0 & 0 & 0 \\
0 & 0
\end{ytableau}\]
\[\mathcal{M}=\left\{
\raisebox{.15in}{\begin{ytableau}
 $ $ & & \\
 $ $ & & \\
 *(gray) &  *(gray)
\end{ytableau}\hspace{.15in}
\begin{ytableau}
 $ $ & & \\
 $ $ & & \\
  *(gray) &  *(gray)
\end{ytableau}\hspace{.15in}
\begin{ytableau}
 $ $ & & \\
 $ $ &  *(gray) &  *(gray)\\
  *(gray) &  *(gray)
\end{ytableau}\hspace{.15in}
\begin{ytableau}
 $ $ & & \\
 $ $ & & \\
 $ $ &  *(gray)
\end{ytableau}\hspace{.15in}
\begin{ytableau}
 $ $ & &  *(gray)\\
 $ $ &  *(gray) & *(gray) \\
 $ $ &  *(gray)
\end{ytableau}}
\right\}
\]
\end{example}

One can also easily describe the inverse algorithm, which shows that the Hillman--Grassl correspondence is a bijection between reverse plane partitions of shape $\lambda$ and multisets of rim hooks of $\lambda$.

\section{Sulzgruber insertion}\label{sec:Sulz}
In \cite{sulzgruber2016building}, Sulzgruber defines a bijection between reverse plane partitions of shape $\lambda$ and multisets of rim hooks of $\lambda$ in the form of an insertion and reverse insertion procedure. We now follow his exposition to review how to insert a rim hook $h$ into a reverse plane partition $\pi$ of shape $\lambda$. We will think of $h$ as a rigid object consisting of 3-dimensional cubes and of $\pi$ as stacks of cubes lying on top of shape $\lambda$. For further details, see \cite{sulzgruber2016building}.

First try placing $h$ on top of $\pi$ such that the cubes of $h$ are above their corresponding cells of $\lambda$. If the stacks of cubes below $h$ all have the same height, we are done. This is because placing $h$ as a rigid object on top of $\pi$ yields a reverse plane partition; there are no empty spaces within any stack of cubes. 

If the result is not a reverse plane partition, let $j$ denote the height of the shortest stack of cubes lying below $h$. Break off the segment of $h$ that can be inserted at height $j+1$. Move the remaining portion of $h$ diagonally northeast one cell and start the process over again. Note that multiple cuts and shifts may be necessary to complete this process. If the result is a reverse plane partition, then $h$ may be inserted into $\pi$, and we denote the result by $h* \pi$. If this procedure fails to produce a reverse plane partition, then $h$ may not be inserted into $\pi$.

\begin{example}
Suppose we have reverse plane partition $\pi$, and we wish to insert rim hook $h_{(1,1)}$. Placing $h_{(1,1)}$ on top of $\pi$, we see that the shortest stack of boxes below $h_{(1,1)}$ has height 1. We thus break off the segment of $h_{(1,1)}$ lying above stacks of size 1 (in this case, just the northeasternmost block of $h_{(1,1)}$), and insert this block in its place. This changes the 1 in position $(1,3)$ of $\pi$ to a 2. We now shift the remaining blocks of $h_{(1,1)}$ northeast one box. We notice that the remaining blocks of $h_{(1,1)}$ are no longer contained in $\pi$. This means the insertion has failed, and $h_{(1,1)}$ may not be inserted into $\pi.$

\begin{center}
$\pi=$
\begin{ytableau}
0 & 0 & 1 \\ 0 & 2 & 2 \\ 3 & 3 \\ 3
\end{ytableau}\hspace{.3in}
$h_{(1,1)}=$
\begin{ytableau}
$ $ & & *(gray) \\
$ $ & *(gray) &  *(gray) \\
 *(gray) &  *(gray) \\
  *(gray)
\end{ytableau}\hspace{.3in}
$h_{(1,2)}=$
\begin{ytableau}
$ $ & & *(gray) \\
$ $ & *(gray) &  *(gray) \\
$ $ &  *(gray) \\
$ $
\end{ytableau}\hspace{.3in}
$h_{(1,2)}*\pi=$
\begin{ytableau}
1 & 1 & 2 \\ 1 & 2 & 2 \\ 3 & 3 \\ 3
\end{ytableau}
\end{center}

Suppose instead we wish to insert rim hook $h_{(1,2)}$ into $\pi$. We again place $h_{(1,2)}$ on top of $\pi$, break off the northeasternmost cell, and shift the remaining three blocks of $h_{(1,2)}$ northeast one box. Each of the three blocks now lies over a stack of size 0. We can thus insert these blocks in their current position to obtain $h_{(1,2)}* \pi$. 
\end{example}

Sulzgruber defines one insertion order that has the property that given rim hooks $h_1\leq h_2\leq\ldots\leq h_k$ in the shape $\lambda$, the insertion of $h_k$ into $h_{k-1}*\cdots h_2*h_{1} *0$ always yields a reverse plane partition, where 0 denotes the reverse plane partition where each entry is 0. (Note that we have changed conventions slightly from that in \cite{sulzgruber2016building}.)  After fixing $\lambda$, the \textit{insertion order for $\lambda$} is as follows: First insert some non-negative number of copies of the rim hook $h_{(1,1)}$. Next insertion some non-negative number of copies of the rim hook $h_{(2,1)}$. Continue until you reach the bottom of the first column, and then start again from the top of the second column. Continue in this manner until you have inserted some non-negative number of copies of the rim hook corresponding to the bottom cell of the rightmost column of $\lambda$. This insertion order is illustrated for shape $\lambda=(5,4,4,3,2)$ below.

\begin{center}
\begin{ytableau}
1 & 6 & 11 & 15 & 18\\
2 & 7 & 12 & 16\\
3 & 8 & 13 & 17\\
4 & 9  & 14\\
5 & 10
\end{ytableau}
\end{center}

\begin{theorem}\cite[Theorem 3]{sulzgruber2016building}
Let $\lambda$ be a partition and $\pi$ be a reverse plane partition of shape $\lambda$. Then there exists a unique sequence $h_1,\ldots,h_s$ of rim hooks of $\lambda$ such that $h_i\leq h_{i+1}$ for $i\in[s-1]$ and $\pi=h_s*\cdots*h_2 * h_1*0$
\end{theorem}

\begin{remark}One can easily check that Sulzgruber insertion is not equivalent to the Hillman--Grassl correspondence by looking at the reverse plane partition of shape $(2,2)$ with each box filled with 1. Hillman--Grassl associates this reverse plane partition with multiset $\{h_{(1,1)},h_{(2,2)}\}$ while Sulzgruber associates it with $\{h_{(2,1)}, h_{(1,2)}\}$. 
\end{remark}

\iffalse
\begin{example} Let $\pi$ be the reverse plane partition below and let $h_{(i,j)}$ denote the rim hook corresponding to cell $(i,j)$. Then $\pi = h_{(1,2)}*h_{(3,1)}*h_{(2,1)}*h_{(2,1)}*h_{(1,1)}*0$. The reverse plane partition $h_{(1,2)}*\pi$ is shown on the right. In the process of inserting $h_{(1,2)}$ into $\pi$, we first lay $h_{(1,2)}$ on top of $\pi$. The stack of cubes of $\pi$ in position $(1,3)$ is strictly shorter than any other stack below $h_{(1,2)}$, so the block of $h_{(1,2)}$ in position $(1,3)$ is inserted. The remaining three cubes are shifted one position northwest, where they may all be inserted since they all lie above a stack of size 1.
\begin{center}
$\pi=$ 
\begin{ytableau} 0 & 1 & 2 \\ 1 & 1 & 3 \\ 4 & 4 & 4 \end{ytableau}\hspace{.5in}
$h_{(1,2)}*\pi=$
\begin{ytableau} 0 & 2 & 3 \\ 2 & 2 & 3 \\ 4 & 4 & 4 \end{ytableau}
\end{center}
\end{example}
\fi

\subsection{Sulzgruber reverse insertion}\label{sec:SulzReverse}
Given a reverse plane partition $\pi=h_s*\cdots*h_2*h_1*0$ of shape $\lambda$, we may identify the rim hook $h_s$ using the reverse insertion procedure we now describe.

First, notice that the southwesternmost box of any rim hook can be one of two types of cell of $\lambda$. It is either an \textit{outer corner} of $\lambda$ (i.e., a cell $c$ such that $\lambda-c$ is still a partition), or it is a cell of $\lambda$ lying in the same row as an outer corner and at the bottom of its column. We will categorize any cell that is an outer corner or in the same diagonal as an outer corner as $\mathcal{O}$ and any cell not in $\mathcal{O}$ lying in the same row as an outer corner at the bottom of its column or any cell in the same diagonal as such a cell as $\mathcal{A}$. For example, we have the following categories.
\begin{center}
\begin{ytableau}
$ $ & & \mathcal{A} & \mathcal{A} & \mathcal{O} & \\
\mathcal{O} & & & \mathcal{A} & \mathcal{A} & \mathcal{O} \\
$ $ & \mathcal{O} & \\
\mathcal{O} & & \mathcal{O} \\
\mathcal{A} & \mathcal{O} 
\end{ytableau}
\end{center}

The reverse insertion proceeds as follows:
\begin{enumerate}
\item We scan $\pi$ diagonal by diagonal starting with the northeasternmost diagonal and reading through each diagonal southeast to northwest until we find the first box $b$ such that either
\begin{enumerate}
\item $b\in \mathcal{O}$ and the entry in $b$ is strictly greater than the entry to the left of $b$ or
\item $b\in \mathcal{A}$ and the entry in $b$ is strictly greater than the entry above $b$ and strictly greater than the entry to the left of $b$.
\end{enumerate}
If there is no box above or to the left of $b$, we consider the entry to be 0. Sulzgruber calls any box satisfying $(a)$ or $(b)$ a \textit{candidate}. Box $b$ is the first candidate in the described diagonal reading order. 

\item We inductively form a path $P$ in $\lambda$ starting at $b$ and ending at a cell on the southeast border of $\lambda$ that travels either one box right or one box up at each step  using the following rules. Suppose we have constructed the path up to some box $b'\in\lambda$.
\begin{enumerate}
\item If there is a box of $\lambda$ above $b'$ and its entry is equal to the entry in $b'$, the path travels next to the box above $b'$.
\item If there is a box of $\lambda$ to the right of $b'$ and the box above $b'$ either does not exist or has entry strictly smaller than the entry in $b'$, the path travels to the box to the right of $b'$.

\end{enumerate}
If there is neither a box above $b'$ nor to the right of $b'$, the path ends at $b'$. 
\end{enumerate}

There is then a unique rim hook of $\lambda$ with the same number of boxes as $P$ and the same northeasternmost cell as $P$; this is $h_s$, the last rim hook inserted to obtain $\pi=h_s*\cdots*h_1*0$. We obtain $h_{s-1}*\cdots*h_1*0$ by subtracting one from the entry in each box in $\pi\cap P$.

\begin{remark} This description of reverse insertion differs slightly from that of \cite{sulzgruber2016building} but is easily seen to be equivalent. The key observation is that if the entry above $b'$ is equal to the entry in $b'$, the path must move upward in order for the result of the subtraction to be a reverse plane partition.
\end{remark}

\begin{example} Performing this reverse insertion on $\pi$ below gives the path $P$ highlighted in orange. This path tells us that the last rim hook inserted to obtain $\pi$ was $h_{(1,2)}$. Continuing this process, we see that $\pi = h_{(1,2)}*h_{(3,1)}*h_{(2,1)}*h_{(2,1)}*h_{(1,1)}*0$.

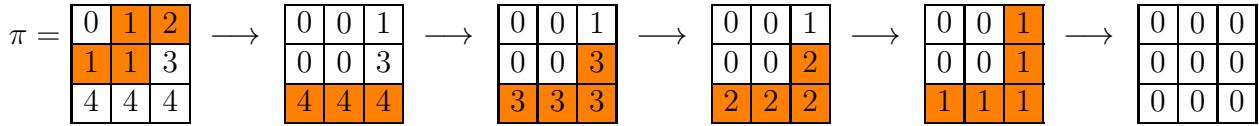
\begin{center}
$\pi=$ 
\begin{ytableau} 0 & *(orange) 1 & *(orange) 2 \\ *(orange) 1 & *(orange) 1 & 3 \\ 4 & 4 & 4 \end{ytableau}{\hspace{.1in}$\longrightarrow$\hspace{.07in}}
\begin{ytableau} 0 & 0 &  1 \\ 0 & 0 &  3 \\ *(orange) 4 & *(orange) 4 & *(orange) 4 \end{ytableau}{\hspace{.1in}$\longrightarrow$\hspace{.07in}}
\begin{ytableau} 0 & 0 & 1 \\ 0 & 0 &  *(orange) 3 \\ *(orange) 3 & *(orange) 3 & *(orange) 3 \end{ytableau}{\hspace{.1in}$\longrightarrow$\hspace{.07in}}
\begin{ytableau} 0 & 0 & 1 \\ 0 & 0 & *(orange) 2 \\ *(orange) 2 & *(orange) 2 & *(orange) 2\end{ytableau}{\hspace{.1in}$\longrightarrow$\hspace{.07in}}
\begin{ytableau} 0 & 0 & *(orange) 1 \\ 0 & 0 & *(orange) 1 \\ *(orange) 1 & *(orange) 1 & *(orange) 1\end{ytableau}{\hspace{.1in}$\longrightarrow$\hspace{.07in}}
\begin{ytableau} 0 & 0 & 0 \\ 0 & 0 & 0 \\ 0 & 0 & 0\end{ytableau}
\end{center}
\end{example}

\section{Diagonal RSK}\label{sec:diagonalRSK}
We next introduce an additional bijection between multisets of rim hooks of $\lambda$ and reverse plane partitions of shape $\lambda$ that uses the RSK algorithm, which we will show is equivalent to Sulzgruber's bijection. We call this bijection \textit{diagonal RSK}. The analogue of this approach for the Hillman--Grassl bijection is described by Gansner in \cite{gansner1977matrix} and also by Morales-Pak-Panova in \cite{morales2015hook}; it is reviewed in Section~\ref{sec:HillmanGrasslRSK}. We phrase our insertion very explicitly to aid in later proofs.

We will use the same insertion order as Sulzgruber, as described in Section~\ref{sec:Sulz}. In addition, we give each rim hook a distinct positive integer label. To obtain the labels, simply fill the boxes of $\lambda=(\lambda_1,\ldots,\lambda_k)$ in increasing order such that row 1 contains exactly $1,2,\ldots,\lambda_1$, row 2 contains exactly $\lambda_1+1,\ldots,\lambda_1+\lambda_2$, etc. Now define the label for rim hook $h_{(i,j)}$ to be the label of cell $(i,j)$ of $\lambda$. For example, the rim hook shown in Figure~\ref{fig:rimhookrppexample} has label 2. 

To obtain a reverse plane partition from a multiset $\mathcal{M}$ of rim hooks of $\lambda$, first order the multiset of rim hooks to be inserted in increasing insertion order. Reading off their labels gives a word, which we call $w(\mathcal{M})$.

We assign a subword of this word to each diagonal as follows. Fix a diagonal $i$. The $j$th letter in the $w(\mathcal{M})$ is included in the subword for diagonal $i$ exactly when the corresponding rim hook has a cell in diagonal $i$. (See Example~\ref{ex:RSKlabels}.) We denote the resulting word by $w(\mathcal{M})_i$.

Next, use RSK insertion to obtain a semistandard Young tableau corresponding to each diagonal: $P(w(\mathcal{M})_i)$. To build the reverse plane partition of shape $\lambda$ corresponding to $\mathcal{M}$, record the shape of $P(w(\mathcal{M})_i)$ in diagonal $i$ by entering the size of the largest part in the southeasternmost cell and proceeding up the diagonal. Given multiset $\mathcal{M}=\{h_1\leq h_2\leq \ldots\leq h_k\}$, we denote the result of this insertion by $h_k \star\cdots\star h_2 \star h_1 \star 0$.

\begin{example}\label{ex:RSKlabels}
Suppose $\lambda=(3,2,2)$ and the multiset of rim hooks is as shown below, where they are written in increasing insertion order. The corresponding multiword is $w(\mathcal{M})=1162777$.
\begin{center}
\ytableausetup{boxsize=.2in}
\raisebox{.15in}{\begin{ytableau}
1 & 2 & 3 \\
4 & 5 \\
6 & 7
\end{ytableau}}\hspace{.05in}
$\mathcal{M}=\left\{
\raisebox{.15in}{\begin{ytableau}
 $ $ & *(gray)& *(gray)\\
 $ $ & *(gray)\\
 *(gray) & *(gray)
\end{ytableau}\hspace{.1in}
\begin{ytableau}
 $ $ & *(gray)& *(gray)\\
 $ $ & *(gray)\\
 *(gray) & *(gray)
\end{ytableau}\hspace{.1in}
\begin{ytableau}
 $ $ & & \\
 $ $ & \\
 *(gray) & *(gray)
\end{ytableau}\hspace{.1in}
\begin{ytableau}
 $ $ & *(gray)& *(gray)\\
 $ $ & *(gray)\\
$ $ & *(gray)
\end{ytableau}\hspace{.1in}
\begin{ytableau}
 $ $ & & \\
 $ $ & \\
 $ $ & *(gray)
\end{ytableau} \hspace{.1in}
\begin{ytableau}
 $ $ & & \\
 $ $ & \\
 $ $ & *(gray)
\end{ytableau}\hspace{.1in}
\begin{ytableau}
 $ $ & & \\
 $ $ & \\
 $ $ & *(gray)
\end{ytableau}}\right\}$
\end{center}
Then $w(\mathcal{M})_{-2}=116$, $w(\mathcal{M})_{-1}=1162777$, and $w(\mathcal{M})_{0}=w(\mathcal{M})_{1}=w(\mathcal{M})_{2}=112$. This gives the five semistandard Young tableaux below, which lead to the reverse plane partition shown on the right.

\begin{center}
\begin{ytableau}
1 & 1 & 6
\end{ytableau}\hspace{.15in}
\begin{ytableau}
1 & 1 & 2 & 7 & 7 & 7 \\ 6
\end{ytableau}\hspace{.15in}
\begin{ytableau}
1 & 1 & 2 
\end{ytableau}\hspace{.15in}
\begin{ytableau}
1 & 1 & 2
\end{ytableau}\hspace{.15in}
\begin{ytableau}
1 & 1 & 2
\end{ytableau}\hspace{.15in}
\begin{ytableau}
0 & 3 & 3 \\
1 & 3 \\
3 & 6
\end{ytableau}
\end{center}
\end{example}

\begin{proposition}\label{prop:RSKgivesRPP}
Suppose $h_1\leq h_2\leq \ldots\leq h_k$ are rim hooks in partition $\lambda$. Then $h_k\star \cdots \star h_2\star h_1 \star 0$ is a reverse plane partition of shape $\lambda$.
\end{proposition}

In order to prove the previous proposition as well as the equivalence of Sulzgruber insertion and diagonal RSK, it will be helpful to examine how $w(\mathcal{M})_d$ is related to $w(\mathcal{M})_{d+1}$. We accomplish this with the following easy lemmas. Recall that $\mathbf{r}_d$ denotes the border box on diagonal $d$.

\begin{lemma}\label{lem:rleft}
Suppose $\mathbf{r}_d$ is directly left of $\mathbf{r}_{d+1}$. Then $w(\mathcal{M})_d$ is a prefix of $w(\mathcal{M})_{d+1}$.
\end{lemma}
\begin{proof}
Every rim hook containing $\mathbf{r}_d$ also contains $\mathbf{r}_{d+1}$, and the rim hooks containing $\mathbf{r}_{d+1}$ but not $\mathbf{r}_d$ are inserted after all of the rim hooks containing $\mathbf{r}_d$.
\end{proof}

\begin{lemma}\label{lem:rbelow}
Suppose $\mathbf{r}_d$ is directly below $\mathbf{r}_{d+1}$. Then $w(\mathcal{M})_{d+1}$ is a subword of $w(\mathcal{M})_{d}$. In particular, the restriction of $w(\mathcal{M})_{d}$ to the interval of labels of rim hooks containing $\mathbf{r}_{d+1}$ gives $w(\mathcal{M})_{d+1}$. 
\end{lemma}
\begin{proof}
In this situation, any rim hook containing $\textbf{r}_{d+1}$ contains $\textbf{r}_d$. In addition, any rim hook that contains $\textbf{r}_{d}$ but not $\textbf{r}_{d+1}$ has label strictly greater than all of the rim hooks containing $\textbf{r}_{d+1}$. The result follows.
\end{proof}

\begin{example}
Consider $\lambda=(3,2,2)$ as in Example~\ref{ex:RSKlabels} and suppose $\mathcal{M}$ contains one copy of each rim hook of $\lambda$. Then we have $w(\mathcal{M})_{-2}=146$ and
$w(\mathcal{M})_{-1}=146257$, so we see $w(\mathcal{M})_{-2}$ as a prefix of $w(\mathcal{M})_{-1}$. We also have $w(\mathcal{M})_{0}=1425$, so we see that $w(\mathcal{M})_{-1}$ restricted to $[1,5]$ gives $w(\mathcal{M})_{0}$.
\end{example}

\begin{proof}[Proof of Proposition~\ref{prop:RSKgivesRPP}]
We first show entries are weakly increasing along rows and down columns. We do this by considering the entries of adjacent diagonals $d$ and $d+1$ of $h_k\star \cdots \star h_2\star h_1 \star 0$. First suppose $\mathbf{r}_d$ is directly left of $\mathbf{r}_{d+1}$. It follows from Lemma~\ref{lem:rleft} that the shape of $P(w(\mathcal{M})_d)$ is contained in the shape of $P(w(\mathcal{M})_{d+1})$, and so the entries in diagonal $d$ are weakly less than the entries to their right in diagonal $d+1$. 

The fact that the entries in diagonal $d$ are weakly greater than the entries above them comes from the following observation. Since $w(\mathcal{M})_d$ is a prefix of $w(\mathcal{M})_{d+1}$, we have that $w(\mathcal{M})_{d+1}=w(\mathcal{M})_d u$ for some possibly empty word $u$. We know that $u$ must be weakly increasing, and part $(a)$ of Lemma~\ref{lem:bumpsmoveleft} says that any RSK bumping path travels weakly left. Thus by part $(b)$ of Lemma~\ref{lem:bumpsmoveleft}, any entry that adds to the length of the $i$th row of $P(w(\mathcal{M})_{d})$ for $i>1$ during the insertion of $u$ must be a letter in the $(i-1)$th row of $P(w(\mathcal{M})_{d})$ that is strictly to the right of the $i$th row of $P(w(\mathcal{M})_{d})$. Thus the size of the $i$th row of $P(w(\mathcal{M})_{d+1})$ cannot exceed the size of the $(i-1)$th row of $P(w(\mathcal{M})_{d})$. It follows that the entries in diagonal $d$ are weakly greater than the entries above them in diagonal $d+1$.

%\textcolor{red}{(what I added applies to the case when $i = 1$)} \textcolor{green}{We don't need this case because the first box of diagonal $d+1$ is not above anything in diagonal $d$.}.

If $\mathbf{r}_d$ is directly below $\mathbf{r}_{d+1}$, it follows from Lemma~\ref{lem:rbelow} that the restriction of $P(w(\mathcal{M})_{d})$ to the letters inserted on diagonal $d+1$ gives $P(w(\mathcal{M})_{d+1})$. This means the entries in diagonal $d$ will be weakly greater than the entries directly above them. 

In addition, the letters in $w(\mathcal{M})_{d}$ not inserted on diagonal $d+1$ are inserted in weakly increasing order. Thus by Lemma~\ref{lem:horizontal strip}, these letters form a horizontal strip in the insertion tableau $P(w(\mathcal{M})_{d})$. It follows that the $i$th row of $P(w(\mathcal{M})_{d})$ must be weakly shorter than the $(i-1)$th row of $P(w(\mathcal{M})_{d+1})$, and so entries in diagonal $d$ will be weakly less than the entries to their right in diagonal $d+1$.

It remains to show that the number of nonzero parts in the shape of $P(w(\mathcal{M})_{d})$ is no larger than the number of boxes of $\lambda$ in diagonal $d$. It follows from Theorem~\ref{thm:GreeneKleitman} that the number of parts of $P(w)$ is equal to the minimal number of disjoint non-decreasing sequences in $w$ necessary to partition the word $w$. Consider some diagonal $j-i$ with border box in position $(i,j)$ and consider the word $\widetilde{w}$ obtained by reading the rim hook labels in boxes \[(1,1),(2,1),\ldots,(i,1),(1,2),(2,2),\ldots,(i,2),\ldots,(1,j),\ldots,(i,j).\] For example, the corresponding word for diagonal -1 in Example~\ref{ex:RSKlabels} is $\widetilde{w} = 146257$. This word can clearly be written as the disjoint union of $j$ non-decreasing sequences: $\{(1,1),\ldots,(i,1)\}$, 
$\{(1,2),\ldots,(i,2)\}$,\ldots,$\{(1,j),\ldots,(i,j)\}$. It can also be written as the disjoint union of $i$ non-decreasing sequences: $\{(1,1),(1,2),\ldots,(1,j)\}$, $\{(2,1),(2,2),\ldots,(2,j)\}$, $\{(i,1),(i,2),\ldots,(i,j)\}$. Hence the minimal number of disjoint non-decreasing sequences it takes to partition the word is at most min$(i,j)$.  Now $w(\mathcal{M})_{j-i}$ can be obtained from $\widetilde{w}$ by deleting entries and repeating existing entries, and can thus still be covered by min$(i,j)$ non-decreasing sequences. 
\end{proof}

We develop the inverse procedure in the next section.

\subsection{Diagonal RSK is equivalent to Sulzgruber insertion}

In this section, we will show that diagonal RSK is equivalent to the Sulzgruber insertion by showing that the reverse insertion procedure is the same in both cases. The fact that the two insertion procedures are equivalent may be initially surprising because the entries on diagonal $d$ of a reverse plane partition seemingly cannot be computed using Sulzgruber's insertion without knowing what happens in diagonal $d+1$, while any one diagonal of a reverse plane partition can clearly be computed on its own in diagonal RSK. Recall that we use $*$ for Sulzgruber insertion and $\star$ for diagonal RSK.

\begin{theorem}\label{thm:diagisSulz}
Suppose $h_1\leq h_2\leq \ldots\leq h_k$ are rim hooks in partition $\lambda$. Then \[h_k \star \cdots\star h_2 \star h_1 \star 0=h_k *\cdots * h_2 * h_1 * 0.\] In other words, diagonal RSK coincides with Sulzgruber's insertion.
\end{theorem}

\begin{proof} To prove the result, we show that the reverse insertion procedure for $\star$ is the same as that of $*$. We first assume for notational simplicity that $\lambda$ is an $(r+1)\times s$ rectangle. In this case, it is easy to label the rim hooks of $\lambda$, as shown below. However, the following argument may be used for any partition $\lambda$.  

\begin{center}
\ytableausetup{boxsize=.6in}
\begin{ytableau}
1 & 2 & \ldots &  \ldots & s\\
s+1 & s+2 & \ldots & \ldots & 2s\\
\vdots & \ddots & \ddots & \ddots & \vdots\\
rs+1 & rs+2 & \ldots & \ldots & (r+1)s
\end{ytableau}
\end{center}

We briefly review the ideas of Lemmas~\ref{lem:rleft} and~\ref{lem:rbelow} in the case of the $(r+1)\times s$ rectangle. Consider a diagonal whose border box is in the bottom row of $\lambda$ in column $j$ such that the diagonal to its right also has border box in the bottom row of $\lambda$. Suppose the word inserted into this diagonal is $w$. Notice then that the word inserted into the diagonal to its right is $wu$, where $u$ is a possibly empty, nondecreasing word in the alphabet $\alpha s +(j+1)$, where $0\leq\alpha\leq r$. 

Next consider a diagonal $d$ whose border box $\mathbf{r}_d$ is in the rightmost column of $\lambda$ in row $i$ such that the diagonal below it $d-1$ also has $\mathbf{r}_{d-1}$ in the rightmost column of $\lambda$. Then the restriction of $w(\mathcal{M})_{d-1}$ to the alphabet $[1,is]$ gives $w(\mathcal{M})_{d}$. In particular, $P(w(\mathcal{M})_{d-1})_{[1,is]}=P(w(\mathcal{M})_{d})$ by Lemma~\ref{lem:restriction}.
%Suppose the word inserted into this diagonal is $h$. Then the restriction of the word inserted into the diagonal below it to the alphabet $[1,is]$ is $h$. In particular, the Young tableau associated to the \textcolor{blue}{word of the(?)} lower diagonal restricted to $[1,is]$ gives the Young tableau associated to the upper diagonal. 

Suppose we are given a rectangular reverse plane partition $\pi=h_k\star \cdots \star h_2 \star h_1 \star 0 $, where $h_1\leq h_2\leq \ldots\leq h_k$ are rim hooks in partition $\lambda$. Let $\mathcal{M}=\{h_1,\ldots,h_k\}$. We wish to recover $h_k$ and $h_{k-1}\star\cdots \star h_2 \star h_1 \star 0 $.

Let $\ell$ denote the label of $h_k$, and suppose rim hook $h_k$ has southwesternmost box in diagonal $d_b$. Let $b$ denote the box in diagonal $d_b$ of $\pi$ corresponding to the row of $P(w(\mathcal{M})_{d_b})$ where the insertion of $\ell$ terminated. (In other words, if  $P(w(\{h_1,\ldots,h_{k-1},h_k\})_{d_b})/P(w(\{h_1,\ldots,h_{k-1}\})_{d_b})$ is a single box in row $r$, then $b$ is the $r$th box from the bottom of diagonal $d_b$ in $\pi$.)

Notice that the entry in box $b$ must be strictly larger than the entry directly to its left. Indeed, since this entry to the left of $b$ was not affected by the insertion of $h_k$, it must be at most one less than the entry in $b$ in order for $h_{k-1} \star \cdots \star h_2 \star h_1 \star 0 $ to have been a reverse plane partition. (Recall here that if there is no box to the left of $b$, we consider the entry to its left to be 0. Similarly, if there is no box above $b$, we consider the entry above $b$ to be 0.)

In addition, as explained in Section~\ref{sec:SulzReverse}, box $b$ must be one of two types. Either $d_b$ contains the unique outer corner of $\lambda$ or $d_b$ contains a border box to the left of the unique outer corner. In the first case, we again say $b$ is of type $\mathcal{O}$, while in the latter case, we say $b$ is of type $\mathcal{A}$. 

If $b$ is of type $\mathcal{A}$, then the entry in box $b$ of $\pi$ must also be strictly greater than the entry directly above it. To see why this is true, suppose the border box $\mathbf{r}_{d_b}$ is in position $(i,j)$ in $\lambda$. Since $h_k$ was the largest rim hook inserted, none of the rim hooks with labels $\alpha s + k$ for $k>j$ and $0\leq\alpha\leq r$ (i.e., the rim hooks whose defining boxes are strictly right of column $j$) have been inserted. Thus $w(\mathcal{M})_{d_b}=w(\mathcal{M})_{d_b+1}$. This implies the insertion of $\ell$ in diagonal $d_b+1$ (the diagonal to the right of $d_b$) terminated in the same row of the insertion tableau as the insertion of $\ell$ in diagonal $d_b$. Thus the insertion of $h_k$ into the reverse plane partition had the effect of adding 1 to $b$ in diagonal $d_b$ and the box directly right of $b$ in diagonal $d_b+1$ but not changing the entry in the box directly above $b$. Thus in $h_{k-1}\star\cdots \star h_2 \star h_1 \star 0 $, box $b$ will have entry one smaller while the box directly above $b$ will have the same entry. It follows that the entry in the box directly above $b$ must be strictly smaller in $\pi$ in order for $h_{k-1}\star\cdots \star h_2 \star h_1 \star 0 $ to be a reverse plane partition. 

We have shown that box $b$ must fit into one of the following two categories:
\begin{enumerate}
\item[$(i)$] box $b$ is of type $\mathcal{O}$ and its entry is strictly greater than the entry to its left in $\pi$, or
\item[$(ii)$] box $b$ is of type $\mathcal{A}$ and its entry is strictly greater than the entry to its left and the entry above it in $\pi$. 
\end{enumerate}
Following Sulzguber's terminology, we call such a box $b$ of $\pi$ a \textit{candidate}.

To begin the reverse insertion of $h_k$, scan the reverse plane partition $\pi$ southeast to northwest along each diagonal starting with the northeasternmost diagonal and locate the first candidate $b$. We next argue that this first candidate $b$ is in diagonal $d_b$ (the southwesternmost diagonal where $h_k$ has support) and is located in the box of diagonal $d_b$ corresponding to the row where the insertion of $\ell$ into $P(w(\{h_1,\ldots,h_{k-1}\})_{d_b})$ ended.

The fact that the entry in box $b$ is strictly greater than the entry to its left means that we have inserted a different word into $d_b$ than we have inserted into $d_b-1$. Suppose the border box $\mathbf{r}_{d_b}$ is in column $j$. This implies we have inserted a nonempty submultiset of the rim hooks $j,s+j,\ldots,rs+j$. The fact that $b$ is the southeasternmost candidate in its diagonal implies that it is the place where the insertion of the largest (in insertion order) of this submultiset terminated by part $(b)$ of Lemma~\ref{lem:bumpsmoveleft}.

We next define a lattice path in $\lambda$ starting at $b$, ending at a box in column $s$ of $\lambda$, and taking upward and rightward steps in $\lambda$ as follows. Suppose we have constructed the path up to some box $b'$.
\begin{itemize}
\item If the entry in $b'$ is equal to the entry above $b'$, the next box in the path is the box above $b'$.
\item If the entry in $b'$ is not equal to the entry above $b'$ and there is a box to the right of $b'$, the next box in the path is the box to the right of $b'$.
\item If the entry in $b'$ is not equal to the entry above $b'$ and there is not a box to the of $b'$, the path terminates at $b'$.
\end{itemize}

We claim that this path exactly identifies the boxes where the insertion of $\ell$ terminated in each relevant diagonal. In other words, subtracting one from each of the entries of the boxes of $\pi$ in the path yields $h_{k-1}\star\cdots \star h_2 \star h_1 \star 0 $ and the boxes in the path uniquely identify $h_k$. 

%We first show that the path defined by inserting any hook can be traveled moving either upward one box or rightward one box at each step. First, since a rim hook covers an interval of diagonals, it is clear that we must include exactly one box from each of these diagonals in such a path. Suppose the insertion of $\ell$ terminated in box $b'$ in some diagonal. 

If the entry in $b'$ is equal to the entry above $b'$, we must include the box in $b'$ in the path in order for $h_{k-1}\star\cdots \star h_2 \star h_1 \star 0 $ to be a reverse plane partition, which we know it is. Therefore the next box in the path is the box above $b'$.

If the entry in $b'$ is not equal to the entry above $b'$ and there is a box to the right of $b'$, we show why the insertion of $\ell$ into diagonal $d_{b'}+1$ must have terminated in the box to the right of $b'$. Suppose first that $\mathbf{r}_{d_{b'}}$ is in the rightmost column. The fact that the entry in the box directly above $b'$ is strictly smaller means that the box from which we begin the reverse insertion of $\ell$ in $P(w(\mathcal{M})_{d_{b'}})$ does not exist in $P(w(\mathcal{M})_{d_{{b'}+1}})$. This means that the entry in this box of $P(w(\mathcal{M})_{d_{b'}})$ is strictly greater than all entries in $P(w(\mathcal{M})_{d_{b'}+1})$. Reverse inserting $\ell$ from $P(w(\mathcal{M})_{d_{b'}})$ and restricting to the alphabet of $d_{b'}+1$ gives the result of reverse inserting $\ell$ from $P(w(\mathcal{M})_{d_{b'}+1})$. Since the entry in the box that starts the reverse insertion in diagonal $d_b'$ is strictly greater than all entries in the tableau for the diagonal above, performing the reverse insertion and restricting will not change the size of the row corresponding to the box above $b'$. Thus the path does not travel upward.

To show we go right in the path, we must show that the number that bumped the entry that ended up in the row of $P(w(\mathcal{M})_{d_{b'}})$ corresponding to $b'$ was also inserted into diagonal $d_{b'}+1$. This means that the insertion terminated one row higher in diagonal $d_{b'}+1$, which means the path should continue to the right. This follows from the fact that the letters inserted into diagonal $d_b'$ but not inserted into diagonal $d_{b'+1}$ form a horizontal strip in $P(w(\mathcal{M})_{d_{b'}})$ with weakly increasing reading word, as shown in the proof of Proposition~\ref{prop:RSKgivesRPP}. Thus they may not bump each other. 

Now if border box $\mathbf{r}_{d_{b'}}$ is in the last row of $\lambda$ but not the last column of $\lambda$, the same argument that we used to show that candidates of type $\mathcal{A}$ must be strictly larger than the entry above them shows that we do not travel upward in our path. Since the word inserted into diagonal $d_{b'}$ is the same as the word inserted into the diagonal to its right, $d_{b'}+1$, it is clear that the insertion of $\ell$ terminated in the same row, and thus we go right in the path.

Once the path is determined, we can identify $h_k$ as the rim hook whose northeasternmost box is the box at the end of the path and that contains the same number of boxes as the path. It is clear that this reverse insertion procedure is the same as that for Sulzgruber's $*$ insertion, and thus the two insertions agree for the rectangle.

It is straightforward to see that this proof can be easily adapted to any shape $\lambda$ using the same arguments (but with less convenient notation) to show the same reverse insertion procedure holds for any shape $\lambda$.
\end{proof}

In the next section, we review the analogous result for Hillman--Grassl, which can be found for example in \cite{morales2015hook}. In this case, the Greene--Kleitman invariant for Hillman--Grassl was known, and the explicit RSK interpretation is obtained as a corollary. 

\section{Hillman--Grassl as RSK}\label{sec:HillmanGrasslRSK}
One can describe the Hillman--Grassl correspondence similarly using a variant of the diagonal RSK defined above. We next remind the reader of such a description, and we give it in the same spirit as diagonal RSK. We will refer to it  as \textit{HG diagonal RSK}. We present this description of Hillman--Grassl to emphasize the fact that both Hillman--Grassl and Sulzgruber insertion are determined by a linear ordering of the cells of the Young diagram (the insertion order) along with a compatible labeling of rim hooks. 

To obtain the Hillman--Grassl insertion order for shape $\lambda$, label the cells of $\lambda$ top to bottom along columns starting with the rightmost column. The labeling for the rim hooks is obtained by labeling the cells of $\lambda$ right to left along rows starting with the top row. See Figure~\ref{fig:HGRSKorder}, where we see, for example, that inside $\lambda=(3,3,2)$, rim hook $h_{(2,1)}$ is seventh in the insertion order and has label 6.
\begin{figure}[h!]
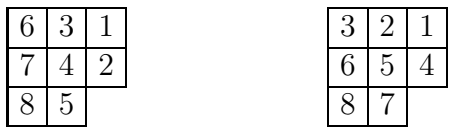

\ytableausetup{boxsize=.2in}
\begin{ytableau}
6 & 3 & 1 \\
7 & 4 & 2 \\
8 & 5
\end{ytableau}\hspace{1in}
\begin{ytableau}
3 & 2 & 1 \\
6 & 5 & 4 \\
8 & 7
\end{ytableau}
\caption{HG diagonal RSK insertion order (left) and rim hook labeling (right).}\label{fig:HGRSKorder}
\end{figure}

Given a multiset of rim hooks $\mathcal{M}$ of $\lambda$, we can again form a word corresponding to each diagonal of $\lambda$ in an analogous way:
\begin{enumerate}
\item First, order the multiset $\mathcal{M}$ by insertion order. 
\item Create a word $w_{HG}(\mathcal{M})$ by recording the corresponding label for each rim hook in the ordered multiset $\mathcal{M}$. 
\item Obtain word $w_{HG}(\mathcal{M})_i$ corresponding to diagonal $i$ by restricting $w_{HG}(\mathcal{M})$ to the labels of rim hooks with support on diagonal $i$. 
\end{enumerate}
As before, we then perform RSK insertion with each word $w_{HG}(\mathcal{M})_i$, and we form the reverse plane partition corresponding to $\mathcal{M}$ by recording the row sizes of the insertion tableau $P(w_{HG}(\mathcal{M})_i)$ on diagonal $i$. 

We now state the Greene--Kleitman invariants for HG diagonal RSK, which will lead to the proof that HG diagonal RSK is equivalent to Hillman--Grassl. Let $B=(b_{ij})$ be a nonnegative integer matrix. An $A_{HG}$-chain in the matrix $B$ is a sequence $V=((i_1,j_1),\ldots,(i_k,j_k))$ of positions in $B$ where $i_1\leq i_2\leq\cdots\leq i_k$ and $j_1\geq j_2\geq\cdots\geq j_k$, where $(i,j)$ can be used at most $b_{ij}$ times. Let $|V|$ denote the cardinality of a $A_{HG}$-chain. Define $a_k^{HG}(B)$ to be the maximum of $|V_1|+\cdots+|V_k|$, where the maximum is over all collections of $k$ $A_{HG}$-chains such that $(i,j)$ is used at most $b_{ij}$ times in the collection and $|V_i|$ denotes the cardinality of $V_i$. Let $a_0^{HG}(B)=0$

Similarly, define a $C_{HG}$-chain in the matrix $B$ to be a sequence $U=((i_1,j_1),\ldots,(i_k,j_k))$ of \textit{distinct} positions in $B$ where $i_1< i_2<\cdots < i_k$ and $j_1< j_2<\cdots< j_k$, where $(i,j)$ can be used only if $b_{ij}\neq 0$. Define $c_k^{HG}(B)$ to be the maximum of $|U_1|+\cdots+|U_k|$, where the maximum is over all collections of $k$ $C_{HG}$-chains such that $(i,j)$ is used at most once in the collection and $|U_i|$ denotes the cardinality of $U_i$. Let $c_0^{HG}(B)=0.$

\begin{remark}
A $C_{HG}$-chain is classically referred to as a $D$-chain, where the ``$D$'' stands for ``descending''. Since we use 
``$d$'' to denote a diagonal of a Young diagram, we use the $C_{HG}$-chain (and $C_S$-chain in the next section) terminology to avoid confusion.
\end{remark}

Given a partition $\lambda$ and a multiset $\mathcal{M}$ of rim hooks of $\lambda$, form an $\mathbb{N}$-tableau $T_\mathcal{M}$ by entering in each box of $\lambda$ the number of times the corresponding rim hook appears in $\mathcal{M}$. For each diagonal $d=j-i$ of $\lambda$, define the matrix $B_d$ to be the $i\times j$-matrix obtained by restricting $T^{HG}_\mathcal{M}$ to the boxes weakly northwest of $\mathbf{r}_d$. 

\begin{theorem}\label{thm:HGGK} Let $\lambda$ be a partition, $\mathcal{M}$ be a multiset of rim hooks of $\lambda$. Let $\pi$ denote the reverse plane partition obtained from $\mathcal{M}$ using HG diagonal RSK, and let $\mu^d$ be the partition obtained by reading the entries of $\pi$ in diagonal $d$. Then
\begin{eqnarray*}
a_k^{HG}(B_d)&=&\mu^d_1+\mu^d_2+\cdots+\mu^d_k\text{  and } \\
c_k^{HG}(B_d)&=&(\mu^d)'_1+(\mu^d)'_2+\cdots+(\mu^d)'_k.\end{eqnarray*} In other words,
\begin{eqnarray*}
\mu^d_i&=&a_i^{HG}(B_d)-a_{i-1}^{HG}(B_d)\text{  and }\\
(\mu^d)'_i&=&c_i^{HG}(B_d)-c_{i-1}^{HG}(B_d)
\end{eqnarray*}
\end{theorem}
\begin{proof}
The result follows immediately from Theorem~\ref{thm:GreeneKleitman} and the fact that $A_{HG}$-chains (resp., $C_{HG}$-chains) in $B_d$ correspond exactly to nondecreasing (resp., decreasing) sequences in $w_{HG}(\mathcal{M})_d$.
\end{proof}

The following corollary is well known and can be found, for example, in \cite{morales2015hook}.

\begin{corollary}
The Hillman--Grassl correspondence is equivalent to HG diagonal RSK.
\end{corollary}
\begin{proof}
As stated in \cite{gansner1977matrix, gansner1981hillman}, the Hillman--Grassl correspondence has the same Greene--Kleitman invariants as given in Theorem~\ref{thm:HGGK}.
\end{proof}

\section{Greene--Kleitman invariants for Sulzgruber insertion}\label{sec:GK}
We now state the Greene--Kleitman invariants for Sulzgruber insertion in the same spirit as those stated for for Hillman--Grassl by Gansner in \cite{gansner1977matrix, gansner1981hillman} and for HG diagonal RSK in Theorem~\ref{thm:HGGK}. 

Let $B=(b_{ij})$ be a nonnegative integer matrix. An $A_S$-chain in the matrix $B$ is a sequence $W=((i_1,j_1),\ldots,(i_k,j_k))$ of positions in $B$ where $i_1\leq i_2\leq\cdots\leq i_k$ and $j_1\leq j_2\leq\cdots\leq j_k$, where $(i,j)$ can be used at most $b_{ij}$ times. Let $|W|$ denote the cardinality of a $A_S$-chain. Define $a_k^S(B)$ to be the maximum of $|W_1|+\cdots+|W_k|$, where the maximum is over all collections of $k$ $A_S$-chains such that $(i,j)$ is used at most $b_{ij}$ times in the collection and $|W_i|$ denotes the cardinality of $W_i$. Let $a_0^S(B)=0.$

Similarly, define a $C_S$-chain in the matrix $B$ to be a sequence $R=((i_1,j_1),\ldots,(i_k,j_k))$ of \textit{distinct} positions in $B$ where $i_1> i_2>\cdots > i_k$ and $j_1< j_2<\cdots< j_k$, where $(i,j)$ can be used only if $b_{ij}\neq 0$. Define $c_k^S(B)$ to be the maximum of $|R_1|+\cdots+|R_k|$, where the maximum is over all collections of $k$ $C_S$-chains such that $(i,j)$ is used at most once in the collection and $|R_i|$ denotes the cardinality of $R_i$. Let $c_0^S(B)=0$

Given a partition $\lambda$ and a multiset $\mathcal{M}$ of rim hooks of $\lambda$, we again form an $\mathbb{N}$-tableau $T_\mathcal{M}$ by entering in each box of $\lambda$ the number of times the corresponding rim hook appears in $\mathcal{M}$. For each diagonal $d=j-i$ of $\lambda$, define the matrix $B_d$ to be the $i\times j$ matrix obtained by restricting $T_\mathcal{M}$ to the boxes weakly northwest of $\mathbf{r}_d$. 

\begin{theorem} Let $\lambda$ be a partition, $\mathcal{M}$ be a multiset of rim hooks of $\lambda$. Let $\pi$ denote the reverse plane partition obtained from $\mathcal{M}$ using Sulzgruber insertion, and let $\mu^d$ be the partition obtained by reading the entries of $\pi$ in diagonal $d$. Then
\begin{eqnarray*}
a_k^S(B_d)&=&\mu^d_1+\mu^d_2+\cdots+\mu^d_k\text{  and } \\
c_k^S(B_d)&=&(\mu^d)'_1+(\mu^d)'_2+\cdots+(\mu^d)'_k.\end{eqnarray*} In other words,
\begin{eqnarray*}
\mu^d_i&=&a_i^S(B_d)-a_{i-1}^S(B_d)\text{  and }\\
(\mu^d)'_i&=&c_i^S(B_d)-c_{i-1}^S(B_d)
\end{eqnarray*}
\end{theorem}
\begin{proof}
The result follows immediately from Theorem~\ref{thm:GreeneKleitman}, Theorem~\ref{thm:diagisSulz}, and the fact that $A_S$-chains (resp., $C_S$-chains) in $B_d$ correspond exactly to nondecreasing (resp., decreasing) sequences in $w(\mathcal{M})_d$.
\end{proof}

\section*{Acknowledgements}
We thank Hugh Thomas for getting us started on this problem and Robin Sulzgruber for helpful exchanges. Rebecca Patrias received support NSERC, CRM-ISM, and the Canada Research Chairs Program. Alexander Garver also received support from NSERC and the Canada Research Chairs Program.

\bibliographystyle{plain}
\bibliography{main.bib}

\end{document}